\documentclass[AFST,Unicode]{cedram}

% ...

% ...

\equalenv{theorem}{theo}
\equalenv{theorem*}{theo*}
\equalenv{corollary}{coro}
\equalenv{lemma}{lemm}
\equalenv{lemma*}{lemm*}
\equalenv{question*}{ques*}

\renewcommand{\R}{\mathbb{R}}
\renewcommand{\C}{\mathbb{C}}

\newcommand{\eps}{\varepsilon}
\newcommand{\m}{\mathrm{m}}
\newcommand{\ip}[2]{\langle #1,#2 \rangle}

\title
[Blaschke-Santal\'o inequality for unconditional log-concave measures]
{A Blaschke-Santal\'o inequality \\ for unconditional log-concave measures}

\alttitle{Une in\'egalit\'e de Blaschke-Santal\'o pour les mesures log-concaves inconditionnelles}

\author{\firstname{Emanuel} \lastname{Milman}}
\address{Department of Mathematics,
Technion - Israel Institute of Technology, Haifa 32000, Israel}
\email{emilman@tx.technion.ac.il}

\author{\firstname{Amir} \lastname{Yehudayoff}}
\address{Department of Computer Science, the University of Copenhagen, Universitetsparken 1,
2100 København, Denmark, and Department of Mathematics, Technion - Israel Institute of Technology, Haifa 32000, Israel}
\email{yehudayoff@technion.ac.il}

\dedicatory{Warmly dedicated to Patrick Cattiaux}

\thanks{The research leading to these results is part of a project that has received funding from the European Research Council (ERC) under the European Union’s Horizon 2020 research and innovation programme (grant agreement No.\ 101001677).
A.Y.\ is partially supported by a DNRF Chair grant.}


\altkeywords{In\'egalit\'e de Blaschke-Santal\'o, inconditionnelle, mesures log-concave}


\begin{document}

\begin{abstract}
The Blaschke-Santal\'o inequality states that the volume product
$|K| \cdot |K^\circ|$ of a symmetric convex body $K \subset \R^n$ is maximized by 
the standard Euclidean unit-ball.
Cordero-Erausquin asked whether the inequality remains true
for all even log-concave measures.
We briefly survey the literature around this question and provide details for the known fact that the inequality holds true for all 
unconditional log-concave measures. 
\end{abstract}

\begin{altabstract}
L'in\'egalit\'e de Blaschke-Santal\'o stipule que le produit volumique $|K| \cdot |K^\circ|$ d'un corps convexe sym\'etrique $K \subset \R^n$ est maximisé par la boule unité euclidienne standard.
Cordero-Erausquin s'est demandé si l'in\'egalit\'e restait vraie
pour toutes les mesures log-concaves paires.
Nous passons bri\`evement en revue la litt\'erature sur cette question et fournissons des précisions sur le fait que l'in\'egalit\'e est vraie pour toutes les mesures log-concaves inconditionnelles.
\end{altabstract}

\maketitle

\section{Introduction}

Let $K \subset \R^n$ be a symmetric convex body (i.e., compact with non-empty interior
so that $K=-K$). We denote the norm whose unit-ball is $K$ by $\|x\|_K := \inf \{ t > 0 : x \in t K \}$.
Let $K^\circ := \{y \in \R^n : \forall x \in K \;\; \ip{x}{y} \leq 1\}$ denote the polar body of~$K$
(i.e.~the unit-ball of the dual norm $\|\cdot\|_K^*$). The Euclidean norm is denoted by $\|\cdot\|_2$, and its unit-ball is denoted by
\[
B := \{x \in \R^n : \|x\|_2 \leq 1\} .
\]
The volume product (sometimes called the Mahler volume) of $K$ is defined to be
\[
P(K) := |K| \cdot |K^\circ|,
\]
where we use $| \cdot |$ to denote the standard $n$-dimensional Euclidean volume.
It is easy to check that the volume product is invariant under invertible linear transformations.
So, by standard compactness arguments it attains maximum and minimum values, and it is natural to study the corresponding extremal bodies. The minimizers are expected to be characterized by the long-standing Mahler conjecture~\cite{mahler1939ubertragungsprinzip}, which we will not discuss here. 
The maximizers are characterized by the celebrated Blaschke–Santal\'o inequality (which we only state in the symmetric case):

\begin{theorem*}[\cite{blaschke1945vorlesungen, santalo1949affine, saint1980volume}]
Let $K \subset \R^n$ be a symmetric convex body.
Then $P(K) \leq P(B)$, 
with equality iff $K$ is a centered ellipsoid. 
\end{theorem*}

The inequality was proved by Blaschke \cite{blaschke1945vorlesungen} for $n \leq 3$ 
and Santal\'o \cite{santalo1949affine} for general $n$. The
characterization of ellipsoids as the equality cases was established by Blaschke and
Santal\'o under certain regularity assumptions on $K$, which
were removed by Saint-Raymond \cite{saint1980volume} (see also Meyer and Pajor \cite{meyer1990blaschke}).
Functional versions of this inequality were proved by
Ball~\cite{ball1986isometric}, by Artstein--Avidan, Klartag and Milman \cite{artstein2004santalo},
by Fradelizi and Meyer~\cite{fradelizi2007some}, by Klartag~\cite{KlartagMarginalsOfInequalities}, 
by Lehec~\cite{lehec2009partitions}, and others. We refer to the comprehensive survey on volume products by Fradelizi, Meyer and Zvavitch \cite{FMZ-VolumeProduct} for additional information.

\medskip

A question put forth and studied by Dario Cordero-Erausquin~\cite{cordero2002santalo} pertains to a conjectural strengthening of the Blaschke--Santal\'o inequality. 
For a symmetric convex body $L \subset \R^n$, 
the volume product over $L$ is defined to be
$$P_L(K) := |K \cap L| \cdot |K^\circ \cap L|.$$
Is it true that for every symmetric convex body $K \subset \R^n$,
\begin{equation}
\label{eqn:genS}
P_L(K) \leq P_L(B)?
\end{equation}
Cordero-Erausquin showed in \cite{cordero2002santalo} that the answer is positive
when $K$ and $L$ are \emph{circled} convex bodies in $\C^n$, namely the unit-balls of norms over $\C$, and in addition $L$ is invariant under complex conjugation. 

Note that the functional $P_L(\cdot)$ is no longer invariant under linear transformations.
In particular, dilations may change the value of $P_L(\cdot)$.
Nevertheless, the radius of the maximizing ball $B$ is conjectured to be $1$ so that $B^\circ = B$.
In addition, the maximizers in~\eqref{eqn:genS} are generally not unique;
for example, it always holds that $P_L(B) \leq P_L( B \cap L)$.

\medskip

It is natural to extend~\eqref{eqn:genS} to log-concave measures. An absolutely continuous Borel measure $\mu = f_\mu(x) dx$ on $\R^n$ is called log-concave if $\log f_\mu : \R^n \rightarrow \R \cup \{ -\infty \}$ is a concave function. By the Pr\'ekopa--Leindler theorem and Borell's characterization \cite{borell1975convex}, this is equivalent to the property that $\mu( (1-\lambda) A_0+ \lambda A_1) \geq \mu(A_0)^{1-\lambda} \mu(A_1)^{\lambda}$ for all compact $A_0,A_1 \subset \R^n$ and $0 < \lambda <1$. For example, the uniform measure on a convex body is log-concave. 
The measure $\mu$ is called even if its density $f_\mu$ is even, or equivalently, if  $\mu(A) = \mu(-A)$ for every Borel set $A \subset \R^n$.
The volume product over $\mu$ is defined to be
\[
P_\mu(K) := \mu(K ) \cdot \mu(K^\circ) .
\]

\begin{question*}[Cordero-Erausquin \cite{cordero2002santalo}]
Is it true that for every even log-concave measure $\mu$
and for every symmetric convex body $K$ in $\R^n$,
\begin{equation}
\label{eqn:genSmu}
P_\mu(K) \leq P_\mu(B)?
\end{equation}
\end{question*}

\medskip

It was shown by Klartag in \cite{KlartagMarginalsOfInequalities} that (\ref{eqn:genSmu}) holds for even measures $\mu$ (not necessarily log-concave!) having density of the form $f_\mu(x) = \int_0^\infty t^{n+1} e^{-\alpha t^2} e^{-\Psi(t x)} dt$ where $\Psi : \R^n \rightarrow (-\infty,\infty]$ is an even convex function and $\alpha > 0$. 

\smallskip

Question (\ref{eqn:genSmu}) is related to two other important conjectures pertaining to even log-concave measures $\mu$ and symmetric convex bodies $K$. The first is the generalized~(B) conjecture, stating that $\R \ni t \mapsto \log \mu(e^t K)$ is concave; a stronger version states the same after replacing $e^t$ by $e^{tD}$ where $D$ is any diagonal matrix. When $\mu$ is the Gaussian measure, the strong version was confirmed by Cordero-Erausquin, Fradelizi and Maurey \cite{cordero2004b}, following a question of Banaszczyk \cite{latala2003some}. In \cite{CorderoRotem-BConjForRotations}, using some sharp weighted Poincar\'e inequalities for even probability measures which are log-concave with respect to a rotationally invariant measure, Cordero-Erausquin and Rotem confirmed the strong generalized (B) conjecture for a large class of rotationally-invariant measures $\mu$, and in particular for the uniform measure on $B$.
Reversing the roles of body and measure, they deduced that $\R \ni t \mapsto \mu(e^t B)$ is log-concave for any even log-concave measure $\mu$ (and by applying a linear transformation, the same holds with $B$ replaced by general centered ellipsoids). As $(e^t B)^{\circ} = e^{-t} B$, this immediately yields a positive answer to (\ref{eqn:genSmu}) for $K$'s which are centered Euclidean balls (of arbitrary radius). 
Note that Cordero-Erausquin and Rotem have shown in \cite{CorderoRotem-StrongBFalse} that the strong generalized (B) conjecture is false for general even log-concave $\mu$'s, even for certain non-isotropic Gaussian measures (having non-identity covariance), but the non-strong version remains plausible and widely believed in full generality. 

\smallskip

A more distantly related conjecture is the log-Brunn-Minkowksi conjecture of
B\"{o}r\"{o}czky, Lutwak, Yang and Zhang \cite{boroczky2012log}, involving two symmetric convex bodies $K,L \subset \R^n$; 
as the precise formulation is not essential for the rest of this note, we only refer to ~\cite{boroczky2012log,EMilman-IsomorphicLogMinkowski} for further context and explicit description. 
This conjecture has been established in the plane in \cite{boroczky2012log}, but the general case remains open for $n \geq 3$. 
Some partial results confirm the conjecture for zonoids \cite{VanHandel-LogMinkowskiForZonoids}, under various symmetry assumptions \cite{BoroczkyKalantz-LogBMWithSymmetries,Rotem-logBM,saroglou2015remarks}, for perturbations of the Euclidean ball \cite{ChenEtAl-LocalToGlobalForLogBM} (following \cite{ColesantiLivshyts-LocalpBMUniquenessForBall,CLM-LogBMForBall,KolesnikovEMilman-LocalLpBM}), locally for perturbations of the unit-balls of $\ell_p^n$ \cite{KolesnikovEMilman-LocalLpBM}, and under curvature pinching estimates \cite{EMilman-IsomorphicLogMinkowski, IvakiEMilman-LpMinkowskiUnderCurvaturePinching}. 
The relevance to this note was expounded in the work of Saroglou, who showed in \cite{saroglou2015remarks,Saroglou-logBM2} that the log-Brunn-Minkowski conjecture implies the generalized (B) conjecture, and moreover, that the validity of the log-Brunn-Minkowski conjecture in $\R^n$ for all $n$ is equivalent to the strong generalized (B) conjecture when $\mu$ is the uniform measure on the cube in $\R^n$ for all $n$. 

\smallskip
This makes a link to the topic of this note, since cubes are the quintessential \emph{unconditional} convex bodies. 
Fixing an orthonormal basis in $\R^n$ and corresponding coordinates, a body $K$ is called unconditional if for every $x \in K$,
we have $(\eps_1 x_1,\ldots,\eps_n x_n) \in K$ for all $\eps \in \{\pm 1\}^n$.
This leads to wondering if~\eqref{eqn:genSmu} may be shown to hold, if not for general even log-concave measures $\mu$, at least for unconditional ones, whose density $f_\mu$ is invariant under the aforementioned action of multiplication by $\{\pm 1\}^n$. 

\smallskip

Cordero-Erausquin, Fradelizi and Maurey showed in \cite{cordero2004b} that the strong generalized (B) conjecture holds when \emph{both} $\mu$ and $K$ are unconditional, and Saroglou deduced from this the validity of the log-Brunn-Minkowski conjecture when both bodies $K,L$ are unconditional \cite{saroglou2015remarks}. For unconditional convex bodies $K,L \subset \R^n$,
define the (unconditional and convex) body
$$K^{\frac{1}{2}} L^{\frac{1}{2}}
= \big\{ z \in \R^n : \exists x \in K , y \in L \ \ \forall i \in [n] \
|z_i| = |x_i|^{\frac{1}{2}} |y_i|^{\frac{1}{2}} \big\}.$$

\begin{prop*}[Proposition 8 or 10 in \cite{cordero2004b}]
Let $\mu$ be an unconditional log-concave measure on $\R^n$,
and let $K,L \subset \R^n$ be unconditional convex bodies.
Then,
$$\mu(K) \mu(L) \leq \mu^2(K^{\frac{1}{2}} L^{\frac{1}{2}}) .$$
\end{prop*}

This proposition immediately leads to the following known corollary (see e.g.~\cite[Corollary 2]{fradelizi2007some}),
which answers~\eqref{eqn:genSmu} in the case that both $\mu$ and $K$
are unconditional. In a sense, this is the analogue over $\R$ of a result of  Cordero-Erausquin \cite{cordero2002santalo}, who showed that the answer to (\ref{eqn:genSmu}) is positive when $\mu$ is a log-plurisubharmonic measure on $\C^n$ invariant under complex conjugation, and both $\mu$ and $K$ are circled
(i.e.~invariant under the action of mutliplication by $e^{i \theta}$).

\begin{corollary}
\label{cor:BSu.c.}
Let $\mu$ be an unconditional log-concave measure on $\R^n$,
and let $K \subset \R^n$ be an unconditional convex body.
Then $P_\mu(K) \leq P_\mu(B)$. 
\end{corollary}

\begin{proof}
Let $L = K^\circ$ be the polar of $K$. Because $K$ is unconditional, so is $L$.
It follows that $K^{\frac{1}{2}} L^{\frac{1}{2}} \subseteq B$, because for all $x \in K$ and $y \in L$, if $|z_i| = |x_i|^{\frac{1}{2}} |y_i|^{\frac{1}{2}}$ then
\[
\sum_{i=1}^n z_i^2 = \sum_{i=1}^n |x_i| |y_i| \leq \|x\|_K \|y\|_L \leq 1 .
\]
Consequently,
\begin{equation*}
P_\mu(K)
= \mu(K) \mu(L) \leq \mu^2(K^{\frac{1}{2}} L^{\frac{1}{2}})
\leq \mu^2(B) = P_\mu(B). \qedhere
\end{equation*}
\end{proof}

\bigskip

The main impetus for writing this note is the less obvious consequence that~\eqref{eqn:genSmu} holds true when $\mu$ is assumed to be unconditional, but without imposing any further requirements on~$K$.  

\begin{theorem}
\label{thm:main}
Let $\mu$ be an unconditional log-concave measure on $\R^n$,
and let $K \subset \R^n$ be a symmetric convex body.
Then $P_\mu(K) \leq P_\mu(B)$. 
\end{theorem}

In particular, this applies to any \emph{rotationally-invariant} log-concave measure $\mu$. In fact, Theorem \ref{thm:main} is known to hold for a more general class of rotationally-invariant measures $\mu = \exp(-\varphi(|x|)) dx$, with the property that $\R \ni t \mapsto \varphi(e^t)$ is convex and non-decreasing (which is weaker than log-concavity). Indeed, by \cite[Corollary 5.4]{CFPP-EasyBusemannAndCampiGronchi}, whenever $\varphi(t)$ is non-decreasing then $\mu(K^\circ) \leq \mu(B_K^\circ)$, where $B_K$ is the centered Euclidean ball so that $|B_K| = |K|$. Since clearly $\mu(K) \leq \mu(B_K)$ by the ``bathtub principle", it follows that $P_\mu(K) \leq P_\mu(B_K)$. It remains to note that $P_\mu(r B) \leq P_\mu(B)$ for all $r > 0$ whenever $\varphi(e^t)$ is convex, because in that case $\R \ni t \mapsto \mu(e^t B)$ is known to be log-concave (e.g.~by \cite{cordero2004b} or \cite{CorderoRotem-BConjForRotations}). 

\subsection{Afterthoughts}
After finishing writing this note, carefully inspecting the literature, and receiving a referee report, we realized three things which we were previously unaware of: Firstly, in the case when $\mu$ is the uniform measure on a convex unconditional body $T$, Theorem \ref{thm:main} was already known to hold 
according to Klartag \cite[p.~135, ll.~6--8]{KlartagMarginalsOfInequalities} (who also accredits this independently to Barthe and Cordero-Erausquin in a private communication).
Secondly, Theorem \ref{thm:main} was already explicitly stated and proved in an identical manner in a survey by Fradelizi, Meyer and Zvavitch \cite[Theorem 7]{FMZ-VolumeProduct}. 
Thirdly, essentially the same symmetrization argument 
appeared in a recent preprint by Colesanti, Livshyts, Kolesnikov and Rotem (see \cite[Section 4 and Theorem 5.7]{CLKR-WeightedFunctionalBS}). They studied a functional version of the Blaschke--Santal\'o inequality, using the volume product
\[
P_{\mu}(\Phi) = \int \exp(-\Phi) d\mu \int \exp(-\Phi^*) d\mu ,
\]
where $\Phi^*(y) := \sup_{x \in \R^n} \langle x,y \rangle - \Phi(x)$ is the Legendre conjugate of the convex function $\Phi$ (in fact, they considered a more general version, where one allows general $L^{p_i}(\mu_i)$ norms with respect to two different measures $\mu_1,\mu_2$ and exponents $p_1,p_2$ on the right). Note that when $\Phi_{K,p}(x) = \frac{1}{p} \| x \|_K^p$ and $p \in [1,\infty]$ then $\Phi_{K,p}^*(x) = \frac{1}{q} \|x\|_{K^{o}}^q$ with $\frac{1}{p} + \frac{1}{q} = 1$. The functional analogue of (\ref{eqn:genSmu}) was resolved by Klartag in \cite[Theorem 4.2]{KlartagMarginalsOfInequalities} (and independently by Barthe and Cordero-Erausquin), who showed that:
\[
P_\mu(\Phi) \leq P_\mu(\Phi_{B,2}) , 
\]
where $\Phi_{B,2}(x)$ is the self-dual function $\frac{1}{2} \|x\|_2^2$; when $\mu$ is the Lebesgue measure $\m$, this was first established by Ball in \cite{ball1986isometric} (see also \cite{artstein2004santalo}). However, contrary to the case when $\mu=\m$, $P_{\mu}(\Phi_{K,p})$ does not coincide in general (up to constants) with $P_{\mu}(K)$  for any value of $p$, and so we do not see how to obtain Theorem \ref{thm:main} by utilizing $P_{\mu}(\Phi)$. 

For the sake of completeness and in the hope that this would be a service to the community, we decided to keep our note in the form of a short survey,
highlighting the simple idea underlying the proof in the most elementary case. 

\medskip
\noindent \textbf{Acknowledgments.} We thank Alexandros Eskenazis, Matthieu Fradelizi, Bo'az Klartag, Dylan Langharst, Galyna Livshyts and the anonymous referee for their encouragement and for pointing out missing references.


\section{Reducing to the unconditional case}

The following lemma reduces the case of symmetric convex bodies
to the case of unconditional ones.

\begin{lemma}
\label{lem:tou.c}
Let $\mu$ be an unconditional log-concave measure on $\R^n$,
and let $K \subset \R^n$ be a symmetric convex body.
Then, there is an unconditional convex body $L \subset \R^n$ so that
$P_\mu(K) \leq P_\mu(L).$
\end{lemma}

\noindent
Theorem~\ref{thm:main} immediately follows from the lemma
and Corollary~\ref{cor:BSu.c.} because
$$P_\mu(K) \leq P_\mu(L) \leq P_\mu(B).$$

The construction of the unconditional body $L$ in Lemma \ref{lem:tou.c} is via a sequence of Steiner symmetrizations.
The Steiner symmetral $S_u K$ of a convex body $K$ in the direction $u \in \mathbb{S}^{n-1}$
is defined as follows: for every $y \in u^{\perp}$,
\begin{equation}
\label{eqn:Stein}
(S_u K)_y = y + \frac{K_y+(-K_y)}{2},
\end{equation}
where $L_y := (L - y) \cap \R u$ is the (translated) one-dimensional fiber
of $L$ over $y$. 
Steiner symmetrization is a standard tool in proving isoperimetric inequalities,
and it enjoys many useful properties.
A particularly useful property for us is that a sequence of~$n$ Steiner symmetrizations in perpendicular directions
produces an unconditional body (see Lemma 2.3 in~\cite{klartag2003isomorphic}).

\begin{lemma*}
Let $K \subseteq \R^n$ be a convex body.
Let $e_1,\ldots,e_n$ be the standard basis of $\R^n$.
Then,
$$L = S_{e_1} S_{e_2} \ldots S_{e_n} K$$
is convex and unconditional. 
\end{lemma*}

The last piece in the proof of Theorem~\ref{thm:main} is the following proposition. 

\begin{proposition}
\label{clm:1S}
Let $\mu$ be an unconditional log-concave measure on $\R^n$,
and let $K \subset \R^n$ be a symmetric convex body.
If $u$ is one of the vectors in the standard basis then 
$$P_\mu(K) \leq P_\mu(S_u K).$$
\end{proposition}

\begin{proof}
First, we consider the effect of symmetrization on $\mu(K)$.
For all $y \in u^\perp$, the restriction of the density $f_\mu(y + \cdot)$ 
to the line $\R u$ leads to an even log-concave measure $\mu_y$ on the line.
So, using~\eqref{eqn:Stein},
$$\mu_y((S_u K)_y) \geq \mu_y^{\frac{1}{2}}(K_y) \mu_y^{\frac{1}{2}}(-K_y) = \mu_y(K_y).$$
Fubini's theorem implies that
$$\mu(S_u K) = \int_{u^\perp} 
\mu_y((S_u K)_y) dy \geq 
 \int_{u^\perp} 
\mu_y(K_y) dy  = \mu(K).$$

Next, we consider the effect of symmetrization on $\mu(K^\circ)$. 
Meyer and Pajor showed that (proof of Lemma 1 in~\cite{meyer1990blaschke})
for $z \in \R u$,
$$(S_u K)^\circ(z)  \supseteq \frac{K^\circ(z) + K^\circ(-z)}{2},$$
where $L(z) := (L-z) \cap u^\perp$
is an $(n-1)$-dimensional slice at height $z$. 
Because $K$ is symmetric, $K^\circ(-z) = - K^\circ(z)$.
The restriction of the density $f_\mu(z+\cdot)$ 
to the hyperplane $u^\perp$ leads to an even log-concave measure $\mu^z$.
So, 
$$\mu^z((S_u K)^\circ(z)) \geq \mu^z(K^\circ(z))^{\frac{1}{2}} \mu^z(-K^\circ(z))^{\frac{1}{2}} = \mu^z(K^\circ(z)).$$
Fubini's theorem implies that
\begin{equation*}
\mu((S_u K)^\circ) = \int_{\R u} 
\mu^z((S_u K)^\circ(z)) dz \geq  \int_{\R u} 
\mu^z(K^\circ(z)) dz  = \mu(K^\circ). \qedhere
\end{equation*}
\end{proof}

\begin{proof}[Proof of Lemma~\ref{lem:tou.c}]
The body 
$$L = S_{e_1} S_{e_2} \ldots S_{e_n} K$$
is convex and unconditional (in fact, the final symmetrization via $S_{e_1}$ is unnecessary since $K$ is symmetric to begin with). By Proposition~\ref{clm:1S}, 
\begin{equation*}
P_\mu(K) \leq P_\mu(S_{e_n} K)
\leq P_\mu(S_{e_{n-1} } S_{e_n} K) \leq
\ldots \leq P_\mu(L) . \qedhere 
\end{equation*}

\end{proof}

\nocite{*}
\bibliography{BS}

\end{document}